\documentclass{amsart}

\usepackage{amsmath,amsthm,amssymb,mathrsfs,tikz-cd}
\usepackage{mathtools}
\usepackage[psdextra]{hyperref}

\definecolor{dark-red}{rgb}{0.6,0,0}
\definecolor{dark-green}{rgb}{0,0.4,0}
\definecolor{nice-purple}{HTML}{8A0087}
\hypersetup{
  colorlinks,
  citecolor=dark-green,
  linkcolor=dark-red,
  urlcolor=nice-purple,
  pdfauthor={Poly Mathews Jr.},
  pdftitle={Matrix group \textLambda-distributions},
  }

\theoremstyle{plain}

\newtheorem{maintheorem}{Theorem}

\newtheorem{corollary}[subsubsection]{Corollary}

\newtheorem{lemma}[subsubsection]{Lemma}

\theoremstyle{definition}

\newtheorem{example}[subsubsection]{Example}

\newtheorem{definition}[subsubsection]{Definition}
\newtheorem{remark}[subsubsection]{Remark}

\newcommand{\Exp}{\mathrm{Exp}}

\newcommand{\restr}[3][]{#2\csname#1\endcsname|\sb{#3}}
\newcommand{\ul}[1]{\underline{#1}}
\newcommand{\Fil}{\mathrm{Fil}}
\newcommand{\mbb}[1]{\mathbb{#1}}
\newcommand{\Log}{\mathrm{Log}}

\newcommand{\abstracttext}{The $\Lambda$-distribution of a compact matrix group is an invariant in algebraic probability theory that was recently introduced to study zero distributions of function field $L$-functions.
It is encoded by the $\sigma$-moment generating function, a generalization of the Molien series of classical invariant theory.
In this work, we compute the $\sigma$-moment generating functions of finite matrix groups in many new cases.
In particular, we compute the asymptotic $\Lambda$-distributions for the infinite families of Weyl reflection groups of types $B_n/C_n$ and $D_n$, complementing the previously known case of reflection groups of type $A_n$, i.e., symmetric groups.
We also establish a general result relating the shapes of $\sigma$-moment generating functions to the distributions of associated classical random variables, explaining a previous ad hoc observation for independent Gaussians arising from traces of powers on compact classical groups.}

\newcommand{\titletext}{Matrix group \boldmath$\Lambda$-distributions}
\newcommand{\runningtitletext}{Matrix group $\Lambda$-distributions}

\title[\runningtitletext]{\titletext}

\author[Poly Mathews Jr.]{Matthew Bertucci \and Sam Van Blarcom \and Benjamin Glancy \and Sean Howe \and Thomas Madden \and Ben Miller \and Souparna Pal \and Giorgos Papagrigoriou \and Nathan Raikman \and Tien Tran}

\begin{document}

\begin{abstract}
\abstracttext
\end{abstract}

\maketitle

\begingroup
\let\boldmath\relax 
\tableofcontents
\endgroup

\section{Introduction}

Let $G \leq \mathrm{GL}_n(\mathbb{C})$ be a compact matrix group.
For $f$ a symmetric function, we write $X_{G,f}$ for the function on $G$ sending $g \in G$ to the evaluation of $f$ at the eigenvalues of $g$ --- for example, when $f=t_1 + t_2 + t_3 + \ldots$, $X_{G,f}$ sends $g \in G$ to its matrix trace.
For $\Lambda$ the ring of symmetric functions, the \emph{$\Lambda$-distribution of $G$}, introduced in \cite{Howe.RandomMatrixStatisticsAndZeroesOfLFunctionsViaProbabilityInLambdaRings}, is the map $\mu_G: \Lambda \rightarrow \mathbb{Z}$ sending $f$ to $\mathbb{E}[X_{G,f}]$, where $\mathbb{E}$ denotes the expected value for the Haar (i.e., uniform) probability measure on $G$ --- a priori these values are complex numbers obtained by integrating $X_{G,f}$ over $G$, but it follows from character theory that they are in fact integers (see, e.g., \cite[\S4.1]{Howe.RandomMatrixStatisticsAndZeroesOfLFunctionsViaProbabilityInLambdaRings}).  

In \cite[Theorem~A]{Howe.RandomMatrixStatisticsAndZeroesOfLFunctionsViaProbabilityInLambdaRings} it was shown that, for $G$ a compact classical group, the $\Lambda$-distribution of $G$ can be expressed succinctly via the \emph{$\sigma$-moment generating function} of an associated $\lambda$-ring valued random variable $X_G$:
\begin{equation}\label{eq.intro-sigmf} 
\mathbb{E}\left[\Exp_{\sigma}(X_G h_1)\right] = \sum_{\tau} \mathbb{E}[X_{G,h_\tau}]m_\tau.
\end{equation}
The sum on the right side of \eqref{eq.intro-sigmf} is over all partitions $\tau$, the $h_\tau$ are the complete symmetric functions, and the $m_\tau$ are the monomial symmetric functions.
The expression $\Exp_{\sigma}$ appearing on the left side of \eqref{eq.intro-sigmf} is the plethystic exponential in the theory of pre-$\lambda$ rings, which plays a role here analogous to the usual exponential in the study of moment generating functions in classical probability theory.
Rather than defining the plethystic exponential precisely (see \S\ref{ss.prelambda} for this), for now the right hand side of \eqref{eq.intro-sigmf} can be taken as the definition of the $\sigma$-moment generating function. 

\begin{example}\label{example.intro-orthogonal}
    For $\mathrm{O}(n) \leq \mathrm{GL}_n(\mathbb{C})$ the group of $n\times n$ orthogonal matrices, \cite[Theorem~A]{Howe.RandomMatrixStatisticsAndZeroesOfLFunctionsViaProbabilityInLambdaRings} gives
    \begin{equation}\label{eq.intro-orthogonal-limit}
    \lim_{n \rightarrow \infty} \mathbb{E}\left[\Exp_{\sigma}(X_{\mathrm{O}(n)} h_1)\right] = \Exp_\sigma(h_2)=\prod_{1\leq i \leq j}\frac{1}{1-t_i t_j}.
    \end{equation}
    In particular, as $n \rightarrow \infty$, we obtain an analog of the moment generating function $e^{\frac{t^2}{2}}$ of a classical Gaussian distribution of mean zero and variance $1$.
    Note that Diaconis and Shahshahani (\cite[Theorem~4]{DiaconisShahshahani.OnTheEigenvaluesOfRandomMatrices}) previously computed the limiting distributions of traces of powers of Haar random orthogonal matrices as independent Gaussian random variables, and this computation is equivalent to \eqref{eq.intro-orthogonal-limit} by \cite[Corollary~4.4.1]{Howe.RandomMatrixStatisticsAndZeroesOfLFunctionsViaProbabilityInLambdaRings}. We give a new, more natural, interpretation of this equivalence in Example~\ref{example.orthogonal-and-symplectic}.
\end{example}

The main purpose of \cite{Howe.RandomMatrixStatisticsAndZeroesOfLFunctionsViaProbabilityInLambdaRings} was to set up a basic theory of probability in pre-$\lambda$ rings and then use it to compare random matrix $\Lambda$-distributions to $\Lambda$-distributions describing the zeroes of function field $L$-functions.
In particular, the computation of the $\sigma$-moment generating functions for classical groups in \cite[Theorem~A]{Howe.RandomMatrixStatisticsAndZeroesOfLFunctionsViaProbabilityInLambdaRings} was made  in order to compare with the $\sigma$-moment generating functions of families of $L$-functions whose statistics were already expected to have matching asymptotic distributions.
This work was subsequently complemented by:
\begin{itemize} 
\item \cite{BertucciHowe.EquidistributionAndArithmeticLambdaDistributions}, which further developed the methods used to compute the arithmetic $\Lambda$-distributions arising from function field $L$-functions.
\item \cite{Howe.TheNegativeSigmaMomentGeneratingFunction}, which, motivated by an application to moments of central values of function field $L$-functions, explained the relation between the $\sigma$-moment generating function of a pre-$\lambda$ random variable $X$ and its negative $-X$.
\end{itemize}

\subsection{Contributions}
In the present work, we focus on developing the theory of \cite{Howe.RandomMatrixStatisticsAndZeroesOfLFunctionsViaProbabilityInLambdaRings} as applied to matrix groups.
In particular, we are less interested here in applications to $L$-functions and more interested in the general shape of $\sigma$-moment generating functions for compact matrix groups, whether or not they arise as monodromy groups for interesting families of $L$-functions.
We make two main contributions:
\begin{enumerate}
    \item In Theorems~\ref{thm.exact-distributions-of-finite-groups} and \ref{thm.asymptotic-distributions-of-finite-groups}, we compute the $\Lambda$-distributions for several families of finite matrix groups (in \cite{Howe.RandomMatrixStatisticsAndZeroesOfLFunctionsViaProbabilityInLambdaRings}, only the symmetric group $\Sigma_n$ embedded in $\mathrm{GL}_n(\mathbb{C})$ as the permutation matrices is treated).
    In particular, we treat natural matrix realizations of cyclic groups, alternating groups, dihedral groups, Pauli groups, and  Weyl groups of type $B_n/C_n$ and type $D_n$.
    The $\sigma$-moment generating function of a finite matrix group is a generalization of its Molien series (whose role in classical invariant theory is described, e.g., in \cite{Stanley.InvariantsOfFiniteGroupsAndTheirApplicationsToCombinatorics}), and one of the tools we use in our computations is an analog of Molien's formula (Lemma~\ref{lemma.MolienAnalog}).
    These results are discussed in \S\ref{s.lambda-finite-groups}.
    \item In Theorem~\ref{thm.lambda-to-classical}, we describe a relation between the $\sigma$-moment generating functions of pre-$\lambda$ random variables and the moment generating functions of associated classical random variables.
    This result explains and generalizes the connection between the pre-$\lambda$ and classical Gaussian distributions observed in Example~\ref{example.intro-orthogonal} and other related computations, answering the question implicit in \cite[Remark~1.1.6]{Howe.RandomMatrixStatisticsAndZeroesOfLFunctionsViaProbabilityInLambdaRings}.
    These results are discussed in \S\ref{s.lambda-to-classical}.  
\end{enumerate}

\subsection{Organization}
In \S\ref{s.preliminaries} we collect some preliminaries on symmetric functions and probability in (pre-)$\lambda$-rings.
In \S\ref{s.lambda-finite-groups} we state and prove our results on the $\Lambda$-distributions of finite groups, including Theorems~\ref{thm.exact-distributions-of-finite-groups} and \ref{thm.asymptotic-distributions-of-finite-groups} described above.
In \S\ref{s.lambda-to-classical} we establish our results connecting $\Lambda$-distributions and classical distributions, including Theorem~\ref{thm.lambda-to-classical} described above.  

\subsection{Acknowledgements}
This work was carried out as part of the Summer 2025 Polymath Jr. program, which is partially supported by NSF grant DMS-2341670.
We thank the organizers of the program and the members and friends of the Random Matrices and $\Lambda$-Distributions group, including: Ronno Das, Will Dudarov, Matthew Hase-Liu, and Phil Tosteson.
Sean Howe was partially supported by NSF grants DMS-2201112 and DMS-2501816 during the preparation of this work.

\section{Preliminaries}\label{s.preliminaries}

In this section we briefly recall some preliminaries.
Our presentation here follows the preliminary sections of \cite{BertucciHowe.EquidistributionAndArithmeticLambdaDistributions} and indexes to \cite{Howe.RandomMatrixStatisticsAndZeroesOfLFunctionsViaProbabilityInLambdaRings} for more detailed discussions. 

\subsection{Permutations}
We write $\Sigma_n$ for the symmetric group on $n$ elements, i.e., the automorphism group of the set $\{1, \ldots, n\}$. 

\subsection{Partitions}
A partition is a non-increasing sequence of non-negative integers $\tau=(\tau_1, \tau_2, \ldots)$ that is eventually zero. Given a partition $\tau$, we write $|\tau|=\sum \tau_i$ and $||\tau||$ for the number of $\tau_i$ that are non-zero.  
\subsection{Symmetric functions}
We will use the notation for symmetric functions described in \cite[\S2.1]{Howe.RandomMatrixStatisticsAndZeroesOfLFunctionsViaProbabilityInLambdaRings}.
In particular, we write 
\[ \Lambda = \varprojlim_n \mathbb{Z}[t_1, \ldots, t_n]^{\Sigma_n} \]
for the ring of symmetric functions and $h_i$ (resp.\ $e_i$, resp.\ $p_i$) for the $i$th complete (resp.\ elementary, resp.\ power sum) symmetric function.
For $\tau=(\tau_1,\tau_2,\ldots)$ a partition, $h_\tau:=h_{\tau_1}h_{\tau_2}\cdots$, and similarly for $p_\tau$ and $e_\tau$.
The monomial symmetric function $m_\tau$ is the formal sum of all distinct permutations of the monomial $t_1^{\tau_1}t_2^{\tau_2}\cdots$.
We note that
\[ h_1=e_1=p_1=m_{(1,0,0,\ldots)}=t_1+t_2+t_3+\ldots.\]

\subsubsection{}\label{sss.cycle-counting}For $i\geq 1$, the $i$th cycle counting symmetric function $p'_i$ is given by
  \[
    p'_i = \frac{1}{i}\sum_{d\mid i}\mu\Bigl(\frac{i}{d}\Bigr)p_d
  \]
where $\mu$ is the Möbius function and $p_d$ is the $d$th power sum symmetric function.

\subsection{Pre-\texorpdfstring{\boldmath$\lambda$}{\textlambda} rings}
\label{ss.prelambda}Recall from \cite[\S2.2 and \S2.4]{Howe.RandomMatrixStatisticsAndZeroesOfLFunctionsViaProbabilityInLambdaRings} that a pre-$\lambda$ ring $R$ is a ring equipped with a plethystic action of the ring $\Lambda$ of symmetric functions, written $a \circ r$ for $a \in \Lambda$ and $r \in R$, satisfying certain compatibilities we do not recall here (see loc. cit.). A $\lambda$-ring is a pre-$\lambda$ ring satisfying a longer list of compatibilities. 

\begin{example}[{See \cite[\S2.2]{Howe.RandomMatrixStatisticsAndZeroesOfLFunctionsViaProbabilityInLambdaRings}}]
\mbox{}
\begin{enumerate}
    \item $\Lambda$ is a $\lambda$-ring with the usual plethystic operations.
    \item $\mathbb{Z}[\mathbb{C}]$, the monoid ring of finite formal sums of complex numbers $\sum a_j [z_j]$, is a $\lambda$-ring with plethystic action uniquely determined by the formula 
    \[ p_i \circ \sum a_j [z_j] = \sum a_j [z_j^i]. \]
\end{enumerate}
\end{example}

\subsubsection{}
Recall from \cite[\S2.3]{Howe.RandomMatrixStatisticsAndZeroesOfLFunctionsViaProbabilityInLambdaRings} that, for any pre-$\lambda$ ring $R$, we have a natural pre-$\lambda$ ring structure on 
\[ R[[\ul{t}_{\mathbb{N}}]] =\varprojlim_{n} R[[t_1, \ldots, t_n]] \]
extending the pre-$\lambda$ ring structure on $R$ and such that $p_i \circ t_j = t_j^i$.
Writing $\mathfrak{m}$ for the set of elements in $R[[\ul{t}_{\mathbb{N}}]]$ with constant term zero, as in \cite[\S2.5]{Howe.RandomMatrixStatisticsAndZeroesOfLFunctionsViaProbabilityInLambdaRings}, we have the $\sigma$-exponential 
\[ \Exp_{\sigma}: \mathfrak{m} \xrightarrow{\sim} 1+ \mathfrak{m},\; F \mapsto \sum_{k \geq 0} h_k \circ F \]
and its inverse (which exists for formal reasons)
\[ \Log_{\sigma}: 1+ \mathfrak{m} \xrightarrow{\sim} \mathfrak{m}. \] 
All of these constructions can be restricted to the pre-$\lambda$ subring $\Lambda_R^{\wedge} \subseteq R[[\ul{t}_{\mbb{N}}]]$ of symmetric power series. By \cite[Lemma~2.5.4]{Howe.RandomMatrixStatisticsAndZeroesOfLFunctionsViaProbabilityInLambdaRings}, if $R$ is a $\mathbb{Q}$-algebra, then 
\begin{equation}\label{eq.exp-power-sum-expansion} \Exp_{\sigma}(F)=\prod_{i=1}^\infty \exp\left(\frac{1}{i} p_i \circ F\right) \end{equation}
where here $\exp(t)=\sum_{k \geq 0} \frac{t^k}{k!}$ is the classical exponential series. 

\subsubsection{}\label{sss.powers}
Recall that, for $F \in 1 + \Fil^1R[[\ul{t}_{\mathbb{N}}]]$ and $N \in R[[\ul{t}_{\mathbb{N}}]]$, we have an associated pre-$\lambda$ power as in \cite[Definition~2.6.1]{Howe.RandomMatrixStatisticsAndZeroesOfLFunctionsViaProbabilityInLambdaRings}, 
\[F^N := \Exp_{\sigma}(N \cdot\Log_\sigma(F)).\]
If $R$ is a $\mathbb{Q}$-algebra and $F$ lies in $\mathbb{Z}[[\ul{t}_{\mathbb{N}}]]$, then we claim
\begin{equation}\label{eq.power-expansion} F^N=\prod_{i=1}^\infty F(t_1^i, t_2^i, \ldots)^{p_i' \circ N}\end{equation}
where on the right the powers are interpreted not as pre-$\lambda$ powers but rather classically using binomial coefficients, i.e., writing $F(t_1^i, t_2^i, \ldots)=1+a$, we interpret $F(t_1^i, t_2^i, \ldots)^{p_i' \circ N}$ as
\[ (1+a)^{p_i' \circ N} = \sum_{k\geq 0} \binom{p_i' \circ N}{k}a^k =\sum_{k\geq 0}\frac{(p_i' \circ N)(p_i' \circ N -1 )\cdots (p_i' \circ N - k +1)}{k!} a^k.\]
Indeed, the formula \eqref{eq.power-expansion} follows by combining \cite[Remark~2.6.3]{Howe.RandomMatrixStatisticsAndZeroesOfLFunctionsViaProbabilityInLambdaRings} and \cite[Lemma~2.8]{Howe.MotivicRandomVariablesAndRepresentationStabilityIConfigurationSpaces}.

\subsection{Pre-\texorpdfstring{\boldmath$\lambda$}{\textlambda} probability}\label{ss.pre-lambda-prob}
\subsubsection{}
We recall from \cite[Definition~3.1.1]{Howe.RandomMatrixStatisticsAndZeroesOfLFunctionsViaProbabilityInLambdaRings} that a pre-$\lambda$ probability space is a pre\nobreakdash-$\lambda$ ring $R$ equipped with a $\mathbb{Z}$-linear expectation functional $\mathbb{E}: R \rightarrow C$ to another ring $C$ such that $\mathbb{E}[1_R]=1_C$. 

\subsubsection{}
Given a pre-$\lambda$ probability space $(R, \mathbb{E})$, we refer to the elements of $R$ as random variables.
Given a random variable $X \in R$, we recall that the $\Lambda$-distribution of $X$ (\cite[Definition~3.1.2]{Howe.RandomMatrixStatisticsAndZeroesOfLFunctionsViaProbabilityInLambdaRings}) is the $\mathbb{Z}$-linear function $\mu_X: \Lambda \rightarrow C$, $f \mapsto \mathbb{E}[f \circ X]$. The $\sigma$-moment generating function of $X$ (\cite[Definition~3.2.1]{Howe.RandomMatrixStatisticsAndZeroesOfLFunctionsViaProbabilityInLambdaRings}) is
\begin{equation}\label{eq.smg-monomial} \mathbb{E}[\Exp_{\sigma}(Xh_1)] \in \Lambda_C^\wedge \end{equation}
where $h_1=p_1=e_1=t_1+t_2+t_3 + \ldots$ is the first complete, power sum, and elementary symmetric function, $\Exp_{\sigma}(Xh_1)$ is computed in $\Lambda_R^\wedge$ (or equivalently in $R[[\ul{t}_{\mathbb{N}}]]$), and expectation is applied coefficient-wise to produce an element of $\Lambda_C^\wedge$. Explicitly, in the basis of monomial symmetric functions we have
\begin{equation}\label{eq.fmg-monomial} \mathbb{E}[\Exp_{\sigma}(Xh_1)] = \sum_{\tau} \mathbb{E}[h_\tau \circ X]m_\tau.\end{equation}
By \cite[Lemma~3.2.2]{Howe.RandomMatrixStatisticsAndZeroesOfLFunctionsViaProbabilityInLambdaRings}, the $\Lambda$-distribution of $X$ and the $\sigma$-moment generating function of $X$ uniquely determine one another via the Hall inner product.

\subsubsection{}
Recall from \cite[\S3.2]{Howe.RandomMatrixStatisticsAndZeroesOfLFunctionsViaProbabilityInLambdaRings} that, using the powers as in \S\ref{sss.powers}, we may consider other moment generating functions $\mathbb{E}[F^X]$.
In particular, we will consider the falling moment generating function $\mathbb{E}[(1+h_1)^X]$ as in \cite[Example~3.2.3]{Howe.RandomMatrixStatisticsAndZeroesOfLFunctionsViaProbabilityInLambdaRings}. For $c_\tau$ the configuration symmetric functions of \cite[Example~2.6.6]{Howe.RandomMatrixStatisticsAndZeroesOfLFunctionsViaProbabilityInLambdaRings}, we have 
\[ \mathbb{E}[(1+h_1)^X]=\sum_{\tau} \mathbb{E}[c_\tau \circ X]m_\tau. \]
 
\section{\texorpdfstring{$\Lambda$}{\textLambda}-distributions of finite groups}\label{s.lambda-finite-groups}

In this section we compute the explicit $\Lambda$-distributions of some simple matrix groups (Theorem~\ref{thm.exact-distributions-of-finite-groups}), and give an approximate formula for generalized permutation groups with applications to asymptotics in natural families (Theorem~\ref{thm.asymptotic-distributions-of-finite-groups}).
We make our setup precise in \S\ref{ss.matrix-group-dist}, then the main results are stated and discussed in \S\ref{ss.main-comp-results}.
These results are proved in \S\ref{ss.comp-proofs} after introducing some new tools for computing $\Lambda$-distributions in \S\ref{ss.comp-tools} (including our analog of Molien's formula, Lemma~\ref{lemma.MolienAnalog}).

\subsection{Matrix group \texorpdfstring{\boldmath$\Lambda$}{\textLambda}-distributions}\label{ss.matrix-group-dist}
Our main results describe $\Lambda$-distributions in the following setting: let $G \leq \mathrm{GL}_n(\mathbb{C})$ be a finite group.
Then, we can consider the $\lambda$-ring of functions from $G$ to $\mathbb{Z}[\mathbb{C}]$, where the $\lambda$-ring structure is pointwise (cf.\ \cite[Remark~4.1.1]{Howe.RandomMatrixStatisticsAndZeroesOfLFunctionsViaProbabilityInLambdaRings}).
It is equipped with a $\mathbb{C}$-valued expectation $\mathbb{E}$ by 
\[ \mathbb{E}[f]=\frac{1}{\#G}\sum_{g \in G} f(g)_1 \]
where, for $\mathbf{a}=\sum a_i [z_i] \in \mathbb{Z}[\mathbb{C}]$, $\mathbf{a}_1 := \sum a_i z_i \in \mathbb{C}$.
These definitions put us in the setting of \S\ref{ss.pre-lambda-prob}, and we write $X_G$ for the random variable which sends $g \in G$ to the formal sum of its eigenvalues (with multiplicity) in $\mathbb{Z}[\mathbb{C}]$. 

\begin{remark} As in \cite[\S4.1]{Howe.RandomMatrixStatisticsAndZeroesOfLFunctionsViaProbabilityInLambdaRings}, one could replace the $\lambda$-ring of $\mathbb{Z}[\mathbb{C}]$-valued functions on $G$ here with the more abstract $\lambda$-ring $K_0(\mathrm{Rep}\, G)$, the Grothendieck ring of representations of $G$, on which we have a $\mathbb{Z}$-valued expectation sending $[V]$ to the multiplicity of the trivial representation in $V$.
The random variable $X_G$ corresponds to the standard representation $[\mathbb{C}^n]$ in this setup; in particular, this explains why the $\sigma$-moment generating functions appearing below have integer coefficients. 
\end{remark}

\subsection{The main results}\label{ss.main-comp-results}

The following theorem gives exact formulas for the $\Lambda$-distributions of some natural families of finite matrix groups. 

\begin{maintheorem}[Exact computations of finite matrix group $\Lambda$-distributions]\label{thm.exact-distributions-of-finite-groups}\hfill
\begin{enumerate}
    \item For $\mathbb{Z}/n\mathbb{Z}$ embedded in $\mathrm{GL}_1(\mathbb{C})$ by $j \mapsto e^{j \frac{2\pi i}{n}}$, 
    \[ \mathbb{E}[\Exp_{\sigma}(X_{\mathbb{Z}/n\mathbb{Z}} h_1)]=\sum_{\tau} \left(\begin{cases} 1 & \textrm{if $n \mid |\tau|$} \\ 0 & \textrm{otherwise}  \end{cases}\right)m_\tau. \]
    \item For $D_n$ the dihedral group with $2n$ elements, embedded in $\mathrm{GL}_2(\mathbb{C})$ by its usual action on the complex plane, 
     \begin{multline*}
     \mathbb{E}[\Exp_{\sigma}(X_{D_n} h_1)]=\\ \frac{1}{2}\left(\sum_{\tau=(\tau_1, \ldots, \tau_s)} \# \left\{ (r_1, \ldots, r_s)\, |\, 0 \leq r_i \leq \tau_i \textrm{ and } n \mid \Bigl(|\tau| -2\sum r_i\Bigr) \right\} m_\tau \right)\\ + \frac{1}{2} \Exp_{\sigma}(p_2).
     \end{multline*}
    \item For $P_n$ the $n$-qubit Pauli group (see \S\ref{ss.pauli-groups}), embedded in $\mathrm{GL}_{2^n}(\mathbb{C})=\mathrm{GL}((\mathbb{C}^2)^{\otimes n})$ by its standard matrix realization, 
    \begin{multline*}
    \mathbb{E}[\Exp_{\sigma}(X_{P_n}h_1)] = \\
    \sum_{\tau=(\tau_1, \ldots, \tau_s)} m_\tau \cdot
    \begin{cases} \frac{A}{4^n}  & \textrm{ if $|\tau|\equiv 0 \mod 4$, some $\tau_r$ is odd} \\
    B + \frac{A-B}{4^n}  & \textrm{ if $|\tau|\equiv 0 \mod 4$, all $\tau_r$ are even} \\
    0 & \textrm{ if $|\tau|\not\equiv 0 \mod 4$}.   
    \end{cases}\\
    \textrm{where } A=\prod_{r=1}^s \binom{2^n + \tau_r - 1}{\tau_r},\; B=\prod_{r=1}^s \binom{2^{n-1} + \frac{\tau_r}{2} - 1}{\frac{\tau_r}{2}}.
    \end{multline*}
    Equivalently, the coefficient of $m_\tau$ in $\mathbb{E}[\Exp_{\sigma}(X_{P_n}h_1)] $ is zero if $4 \nmid |\tau|$ and otherwise agrees with the coefficient of $m_\tau$ in 
    \[ \frac{1}{4^n}\Exp_{\sigma}(2^n p_1) + \Bigl(1-\frac{1}{4^n}\Bigr)\Exp_{\sigma}(2^{n-1}p_2). \]
\end{enumerate}
\end{maintheorem}

\begin{remark}
    As will be clear from the proof of Theorem~\ref{thm.exact-distributions-of-finite-groups}, if a matrix group contains the scalar matrices corresponding to $n$th roots of unity, then all terms of degree not divisible by $n$ in the $\sigma$-moment generating function are identically zero. 
\end{remark}

\begin{example}\label{example.two-asymptotics}
From Theorem~\ref{thm.exact-distributions-of-finite-groups}-(1) and (2), one can also deduce asymptotic distributions: in the case of cyclic groups, Theorem~\ref{thm.exact-distributions-of-finite-groups}-(1) implies
\[ \lim_{n \rightarrow \infty} \mathbb{E}[\Exp_{\sigma}(X_{\mathbb{Z}/n\mathbb{Z}} h_1)] = 1 \]
where the limit is computed coefficient-wise.
Similarly, in the case of dihedral groups, Theorem~\ref{thm.exact-distributions-of-finite-groups}-(2) gives 
\[ \lim_{n \rightarrow \infty} \mathbb{E}[\Exp_{\sigma}(X_{D_n} h_1)] = \frac{1}{2} (1+\Exp_{\sigma}(p_2)). \]
\end{example}

The following theorem gives an asymptotic computation of the $\Lambda$-distributions of some natural families of groups that are closely related to symmetric groups. 

\begin{maintheorem}[Asymptotic computations of some finite group $\Lambda$-distributions]\label{thm.asymptotic-distributions-of-finite-groups}\hfill
\begin{enumerate}
     \item Let $H\leq \Sigma_n$ be a $k$-transitive subgroup, let $\mu_r$ be the group of $r$th roots of unity in $\mathbb{C}$, and let $G= \mu_r^n \rtimes H$, embedded in $\mathrm{GL}_n(\mathbb{C})$ via generalized permutation matrices.
     Then, modulo the $m_\tau$ with $|\tau|>k$,
    \[ \mathbb{E}[\Exp_{\sigma}(X_G h_1)] \equiv \Exp_{\sigma}(h_r + h_{2r} + h_{3r} + \ldots). \]      
    \item Let $G=W(D_n) \leq \mathrm{GL}_n(\mathbb{C})$ be the Weyl group of the $D_n$ root system, or equivalently the group of generalized permutation matrices $(W \rtimes \Sigma_n) \leq \mu_2^n \rtimes \Sigma_n$ for $W \leq \mu_2^n$ the subgroup of tuples whose product is $1$.
    Then, modulo the $m_\tau$ with $|\tau|>n$, 
    \[ \mathbb{E}[\Exp_{\sigma}(X_G h_1)] \equiv \Exp_{\sigma}(h_2 + h_{4} + h_{6} + \ldots)  + h_n.\] 

\end{enumerate}
\end{maintheorem}

\begin{example}\label{example.alternating}
    The alternating group $A_n \leq \Sigma_n$ is $(n-2)$-transitive.
    In particular, Theorem~\ref{thm.asymptotic-distributions-of-finite-groups}-(1) implies that, as $n \rightarrow \infty$, the asymptotic $\Lambda$-distribution of $A_n$ agrees with the asymptotic $\Lambda$-distribution of $\Sigma_n$ and is asymptotically Poisson (see Example~\ref{example.poisson}). 
\end{example}

\begin{remark}
    The Weyl group of the root system\footnote{We caution the reader that we write $A_n$ both for the usual root system of that name and for the alternating group, and similarly we write $D_n$ both for the usual root system of that name and for the dihedral group. The meaning of the symbol should be clear from the surrounding context.} $A_n$ is the symmetric group $\Sigma_{n+1}$, and the Weyl group of the root systems $B_n$ or $C_n$ is the generalized symmetric group $\mu_2^n \rtimes \Sigma_n$.
    Thus, Theorem~\ref{thm.asymptotic-distributions-of-finite-groups}-(1) includes the computation of the asymptotic $\Lambda$-distributions for each of these infinite families of Weyl groups.
    Combined with Theorem~\ref{thm.asymptotic-distributions-of-finite-groups}-(2), we have thus computed the asymptotic $\Lambda$-distribution for the Weyl groups of each infinite family of classical root systems --- note that, in Theorem~\ref{thm.asymptotic-distributions-of-finite-groups}-(2), the $h_n$ term disappears if we work modulo terms of degree $\geq n$ and thus in both the $B_n/C_n$ and $D_n$ cases the asymptotic $\sigma$-moment generating function is $\Exp_\sigma(h_2 + h_4 + \ldots)$.  
    Only the case of symmetric groups, i.e., the root systems $A_n$, was previously known, where the asymptotic $\sigma$-moment generating function is $\Exp_\sigma(h_1 + h_2 + \ldots)$ (this can also be viewed as a special case of Theorem~\ref{thm.asymptotic-distributions-of-finite-groups}-(1)).
\end{remark}

\begin{remark}
One of the main questions driving this work is: Which natural sequences of matrix groups $G_n$ have $\sigma$-moment generating functions converging to an interesting ``simple'' limit as in \cite[Theorem~A]{Howe.RandomMatrixStatisticsAndZeroesOfLFunctionsViaProbabilityInLambdaRings}?
Example~\ref{example.two-asymptotics} and Theorem~\ref{thm.asymptotic-distributions-of-finite-groups} provide new examples, while Theorem~\ref{thm.exact-distributions-of-finite-groups}-(3) provides an example of a natural family where the $\Lambda$-distributions/$\sigma$-moment generating functions do not converge. 
\end{remark}

\subsection{Tools}\label{ss.comp-tools}

We first recall a simple multiplicity formula. 

\begin{lemma}\label{lemma.multiplicity-formula}
    Let $V$ be a finite dimensional complex vector space and let $G \leq \mathrm{GL}(V)$ be a compact group.
    Then 
    \[ \mathbb{E}[\Exp_{\sigma}(X_G h_1)]=\sum_{\tau} \dim_\mathbb{C} (\mathrm{Sym}^\tau V)^G \cdot m_\tau.\]
\end{lemma}
\begin{proof}
    This follows from the argument at the beginning of the proof of \cite[Theorem~4.2.1]{Howe.RandomMatrixStatisticsAndZeroesOfLFunctionsViaProbabilityInLambdaRings}.
\end{proof}

We now give an analog of Molien's formula in classical invariant theory (see, e.g., \cite[\S2]{Stanley.InvariantsOfFiniteGroupsAndTheirApplicationsToCombinatorics}).
Let $G \leq \mathrm{GL}_n(\mathbb{C})$ be a finite group.
For $g \in G$, we write $[g]$ for the multiset of eigenvalues of $g$, which we may also view as an element of the pre-$\lambda$ ring $\mathbb{Z}[\mathbb{C}]$. 
\begin{lemma}\label{lemma.MolienAnalog}
    For $G \leq \mathrm{GL}_n(\mathbb{C})$ a finite group, the following identity holds in  $\Lambda_{\mathbb{C}}^\wedge$:
    \begin{align*}
    \mathbb{E}[\Exp_{\sigma}(X_G h_1)] 
        &=   \frac{1}{\#G}\sum_{g \in G}\left( \prod_{i \geq 1}\prod_{\zeta \in [g]}\frac{1}{1-\zeta t_i} \right)\\
     &=   \frac{1}{\#G}\sum_{g \in G} \left(\prod_{i \geq 1} \det\left(\frac{1}{1 - t_i g}\right) \right).
 \end{align*}
\end{lemma}

\begin{remark}
    There is a natural homomorphism $\Lambda_{\mathbb{C}}^\wedge \rightarrow \mathbb{C}[[t]]$ given by specializing $t_1$ to $t$ and $t_i$ to $0$ for all $i \geq 2$.
    Under this homomorphism, $m_\tau\mapsto 0$ unless $\tau=(\tau_1, 0, 0 \ldots)$, in which case $m_\tau \mapsto t^{\tau_1}$, and thus the equality of
    Lemma~\ref{lemma.MolienAnalog} specializes to Molien's formula for the trivial representation (see \cite[Theorem~2.1]{Stanley.InvariantsOfFiniteGroupsAndTheirApplicationsToCombinatorics}),
    \[ \sum_{j \geq 0} \dim (\mathrm{Sym}^j V)^G t^j= \frac{1}{\#G} \sum_{g \in G}\det\left(\frac{1}{1-tg}\right).\]
\end{remark}

\begin{proof}[First proof of Lemma~\ref{lemma.MolienAnalog}]
     The equivalence between the two expressions on the right follows from the usual expansion of the characteristic power series of a matrix, so it suffices to prove the first of these agrees with the $\sigma$-moment generating function appearing on the left. 
     
    Because the $\lambda$-ring structure on the ring of $\mathbb{Z}[\mathbb{C}]$-valued functions on $G$ is defined pointwise, $\Exp_{\sigma}(X_G h_1)$ can be viewed as the function on $G$ sending $g$ to $\Exp_{\sigma}([g] h_1) \in \Lambda_{\mathbb{Z}[\mathbb{C}]}$.
    Since $[g]=\sum_{\zeta \in [g]}[\zeta]$, using the multiplicativity of $\Exp_{\sigma}$ and \cite[Lemma~2.2.4]{Howe.RandomMatrixStatisticsAndZeroesOfLFunctionsViaProbabilityInLambdaRings},  
    \begin{align*}
    \Exp_{\sigma}([g] h_1) &= \Exp_{\sigma}\left(\sum_{\zeta \in [g]} [\zeta]h_1\right)\\
    &= \prod_{\zeta \in [g]} \Exp_{\sigma}\left([\zeta]h_1\right)\\
    &= \prod_{\zeta \in [g]} \Exp_{\sigma}\left(\sum_{i \geq 1}[\zeta]t_i\right) \\
     &=\prod_{[\zeta]\in [g]}\prod_{i \geq 1} \Exp_{\sigma}([\zeta]t_i)
     \\
     &=\prod_{[\zeta]\in [g]}\prod_{i \geq 1} \frac{1}{1-[\zeta]t_i}.
     \end{align*}
    Applying the expectation function coefficient-wise, we obtain the first claimed equality.
\end{proof}
\begin{proof}[Second proof of Lemma~\ref{lemma.MolienAnalog}]
The equivalence between the two expressions on the right follows from the usual expansion of the characteristic power series of a matrix, so it suffices to prove the second of these agrees with the $\sigma$-moment generating function appearing on the left. 

By Lemma~\ref{lemma.multiplicity-formula}, 
       \[ \mathbb{E}[\Exp_{\sigma}(X_G h_1)]=\sum_{\tau} \dim_\mathbb{C} (\mathrm{Sym}^\tau V)^G m_\tau.\]
As $\dim_\mathbb{C} (\mathrm{Sym}^\tau V)^G$ is the multiplicity of the trivial representation $\mathrm{triv}$ in $\mathrm{Sym}^\tau V$, using character theory we can write 
\[ \dim_\mathbb{C} (\mathrm{Sym}^\tau V)^G = \langle \chi_{\mathrm{Sym}^\tau V}, \mathrm{triv} \rangle= \frac{1}{\#G}\sum_{g\in G} \chi_{\mathrm{Sym}^\tau V}(g). \]
Because $\mathrm{Sym}^\tau V = \mathrm{Sym}^{\tau_1 V} \otimes \mathrm{Sym}^{\tau_2} V \otimes \ldots$ and characters multiply over tensor products,
\[ \chi_{\mathrm{Sym}^\tau V}(g) = \chi_{\mathrm{Sym}^{\tau_1} V}(g)\chi_{\mathrm{Sym}^{\tau_2} V}(g)\ldots. \]
Expanding, we find
\begin{align*}
\mathbb{E}[\Exp_{\sigma}(X_G h_1)]&= \frac{1}{\#G}\sum_{g \in G} \prod_{i \geq 1} \Biggl(\sum_{j \geq 0} \chi_{\mathrm{Sym}^j V}(g) t_i^j\Biggr)\\&=\frac{1}{\#G}\sum_{g \in G} \Biggl(\prod_{i \geq 1}  \det\left(\frac{1}{1 - t_i g}\right) \Biggr),
\end{align*}
where we have used the usual identification of the traces of a matrix on symmetric powers with the coefficients of its characteristic power series. 
\end{proof}

\subsubsection{Formulas for permutation representations}
Let $G$ be a finite group acting faithfully on a finite set $S$.
Then, for $V=\mathbb{C}[S]$ the associated permutation representation, we can view $G$ as a subgroup of $\mathrm{GL}(V)$ (if we fix an enumeration of the elements of $S$, then, in the resulting identification $V=\mathbb{C}^n$, $G$ is a subgroup of permutation matrices).
For $\tau=(\tau_1, \tau_2, \dots,\tau_k)$ a partition, we write 
\[ \mathrm{Sym}^\tau(S)= \prod_i S^{\tau_i} / \prod_{i} \Sigma_{\tau_i}. \]
In other words, $\mathrm{Sym}^\tau(S)$ is the set of ordered tuples $(T_1, \ldots, T_k)$ where each $T_i$ is a multiset of $\tau_i$ elements of $S$.
We note that the action of $G$ on $S$ induces a natural action of $G$ on $\mathrm{Sym}^\tau(S)$ for any $\tau$. 

\begin{lemma}\label{lemma.permutation-rep}
    Let $G$ be a finite group acting faithfully on a finite set $S$.
    Then, for $V=\mathbb{C}[S]$, viewing $G$ as a subgroup of $\mathrm{GL}(V)\cong \mathrm{GL}_{\# S}(\mathbb{C})$, 
    \[ \mathbb{E}[\Exp_{\sigma}(X_G h_1)]= \sum_\tau (\textrm{the number of $G$-orbits in $\mathrm{Sym}^\tau S$}) \cdot m_\tau. \]
\end{lemma}
\begin{proof}
    We claim the identity follows from Lemma~\ref{lemma.multiplicity-formula}. Indeed, $\mathrm{Sym}^\tau V$ is naturally identified with the permutation representation on $\mathrm{Sym}^\tau(S)$, and in any permutation representation the dimension of the invariants is equal to the number of orbits: if $A$ is the set acted on by $G$, then an orbit $o=\{a_1, \ldots, a_r\} \subseteq A$ corresponds to the trivial representation in $\mathbb{C}[A]$ spanned by the vector $a_1 + a_2 + \ldots + a_r$.
\end{proof}

\subsection{Proofs of Theorems~\ref{thm.exact-distributions-of-finite-groups} and \ref{thm.asymptotic-distributions-of-finite-groups}}\label{ss.comp-proofs}

\subsubsection{Cyclic groups}
\begin{proof}[Proof of Theorem~\ref{thm.exact-distributions-of-finite-groups}-(1)]
We apply Lemma~\ref{lemma.multiplicity-formula}: $\mathrm{Sym}^{\tau}V$ is a one-dimensional vector space on which $\mathbb{Z}/n\mathbb{Z}$ acts by the character $k \mapsto e^{\frac{2 \pi i k|\tau|}{n}}$.
This is the trivial character exactly when $n \mid |\tau|$, giving the result. 
\end{proof}

\subsubsection{Dihedral groups}
\begin{proof}[Proof of Theorem~\ref{thm.exact-distributions-of-finite-groups}-(2)]
We can write $D_n = R \sqcup \mathbb{Z}/n\mathbb{Z}$, where $R$ consists of $n$ reflections and $\mathbb{Z}/n\mathbb{Z}$ is the subgroup of rotations.
Applying Lemma~\ref{lemma.MolienAnalog}, we can compute separately the contributions from $R$ and $\mathbb{Z}/n\mathbb{Z}$.
Each $r \in R$ has characteristic power series $\det\left(\frac{1}{1-tr}\right)=\frac{1}{1-t^2}$, so the contribution from $R$ is given by 
\[ \frac{n}{2n}\prod_{i\geq 1}\frac{1}{1-t_i^2}=\frac{1}{2}\prod_{i \geq 1} \Exp_{\sigma}(t_i^2)=\frac{1}{2}\Exp_{\sigma}\biggl(\sum_{i \geq 1} t_i^2\biggr)=\frac{1}{2}\Exp_{\sigma}(p_2).\]

On the other hand, by applying Lemma~\ref{lemma.MolienAnalog} also to $\mathbb{Z}/n\mathbb{Z} \subseteq D_n \subseteq \mathrm{GL}_2(\mathbb{C})$ and then comparing with Lemma~\ref{lemma.multiplicity-formula}, we find the contribution from $\mathbb{Z}/n\mathbb{Z}$ is 
\[ \frac{1}{2}\sum_{\tau} \dim_{\mathbb{C}} (\mathrm{Sym}^\tau(\mathbb{C}^2))^{\mathbb{Z}/n\mathbb{Z}} \cdot m_\tau. \]
Writing $\chi$ for the character $k \mapsto e^{\frac{2\pi i k}{n}}$ of $\mathbb{Z}/n\mathbb{Z}$ we have $\mathbb{C}^2=\chi \oplus \chi^{-1}$ as a representation of $\mathbb{Z}/n\mathbb{Z}$.
For each $i$, we then have $\mathrm{Sym}^{\tau_i} \mathbb{C}^2=\sum_{a+b=\tau_i, a,b\geq 0} \chi^{a-b}$.
For $\tau=(\tau_1, \ldots, \tau_s)$, the number of copies of the trivial character in $\mathrm{Sym}^{\tau} \mathbb{C}^2= \mathrm{Sym}^{\tau_1}\mathbb{C}^2 \otimes \mathrm{Sym}^{\tau_2}\mathbb{C}^2 \ldots$ is then 
\[ \# \left\{ (r_1, \ldots, r_s)\, |\, 0 \leq r_i \leq \tau_i \textrm{ and } n \mid \Bigl(|\tau| -2\sum r_i\Bigr) \right\}, \]
where the tuple $(r_1, \ldots, r_s)$ here corresponds to tensoring the character spaces for $r_1 + (\tau_1 - r_1) = \tau_1, r_2 + (\tau_2 - r_2)$, etc., whence we obtain the character $\chi^{|\tau|- 2 \sum r_i}$, which is trivial if and only if the exponent is divisible by $n$. 

Combining the contributions from $R$ and $\mathbb{Z}/n\mathbb{Z}$ yields the claimed formula. 
\end{proof}

\subsubsection{Pauli groups}\label{ss.pauli-groups}
The $n$-qubit Pauli group is the subgroup of $\mathrm{GL}(\mathbb{C}^2)^{\otimes n})$ consisting of the matrices 
\begin{multline*}
u (M_1 \otimes M_2 \otimes \ldots \otimes M_n) \textrm{ where } u \in \{\pm 1, \pm i\} \textrm{ and } M_j \in \{I, X, Y, Z\} \textrm{ for } \\ 
 I=\begin{bmatrix}1 & 0 \\ 0 & 1 \end{bmatrix},\; X = \begin{bmatrix} 0 & 1 \\ 1 & 0\end{bmatrix},\; Y = \begin{bmatrix} 0 & i \\ -i & 0\end{bmatrix},\textrm{ and } Z=\begin{bmatrix}1 & 0 \\ 0 & -1\end{bmatrix}.
 \end{multline*}

\begin{proof}[Proof of Theorem~\ref{thm.exact-distributions-of-finite-groups}-(3)]
With notation as above, we claim that if any $M_i \neq I$, then $M_1 \otimes M_2 \otimes \ldots \otimes M_n$ has characteristic power series $\left(\frac{1}{1-t^2}\right)^{2^{n-1}}$: to see this, first note that each of $X, Y, $ and $Z$ has eigenvalues $\pm 1$, while $I$ has eigenvalue $1$ with multiplicity $2$, and the eigenvalues of the tensor product are given by the $2^n$ possible products given by selecting one of the eigenvalues of each $M_i$.
If at least one of the matrices is not $I$, then this gives $2^{n-1}$ products of $1$ and $2^{n-1}$ products of $-1$, so that the characteristic power series is $(\frac{1}{1-t})^{2^{n-1}}(\frac{1}{1+t})^{2^{n-1}}=\left(\frac{1}{1-t^2}\right)^{2^{n-1}}$ as claimed.
If $M_1=M_2=\ldots=M_n=I$ then the characteristic power series is $(\frac{1}{1-t})^{2^n}$.
Multiplying by a scalar $u$ has the effect of replacing $t$ with $ut$ in these characteristic power series so, applying Lemma~\ref{lemma.MolienAnalog}, we find
\begin{multline*}
\mathbb{E}[\Exp_{\sigma}(X_{P_n} h_1)]=\\\frac{1}{4^{n+1}}\sum_{u \in \{\pm 1, \pm i\}} \left( \prod_{i}\left(\frac{1}{1-ut_i}\right)^{2^n} + (4^{n}-1)\prod_{i}\left(\frac{1}{1-(ut_i)^2}\right)^{2^{n-1}}\right).
\end{multline*}
Because everything is homogeneous in $u$, evaluating the outer sum yields that the terms of degree (in the variables $t_i$) not divisible by $4$ are identically zero, while the terms of degree divisible by $4$ agree with those of
\begin{multline*}
\frac{1}{4^n}\left( \prod_{i}\left(\frac{1}{1-t_i}\right)^{2^n} + (4^{n}-1)\prod_{i}\left(\frac{1}{1-t_i^2}\right)^{2^{n-1}}\right) \\
\begin{aligned}
&= \frac{1}{4^n}\Exp_{\sigma}(p_1)^{2^n} + \Bigl(1-\frac{1}{4^n}\Bigr)\Exp_{\sigma}(p_2)^{2^{n-1}} \\
&= \frac{1}{4^n}\Exp_{\sigma}(2^n p_1) + \Bigl(1-\frac{1}{4^n}\Bigr)\Exp_{\sigma}(2^{n-1}p_2).
\end{aligned}
\end{multline*}
The explicit coefficient formulas follow from this computation after substituting
\begin{align*} \left(\frac{1}{1-t_i}\right)^{2^n} = \sum_{k \geq 0} \binom{2^n+k-1}{k}t_i^k \;\textrm{ and } 
\left(\frac{1}{1-t_i^2}\right)^{2^{n-1}} = \sum_{k \geq 0} \binom{2^{n-1}+k-1}{k}t_i^{2k}.  
\end{align*}
\end{proof}

\subsubsection{Generalized permutation groups}

\begin{proof}[Proof of Theorem~\ref{thm.asymptotic-distributions-of-finite-groups}-(1)]
We fix a $k$-transitive subgroup $H \leq \Sigma_n$ (in particular, $k \leq n$); for clarity of exposition, we first treat the case where $r=1$, so we are considering the permutation action of $G=H$ on $\mathbb{C}^n$.
Let $\tau=(\tau_1, \ldots, \tau_m)$ with $|\tau| \leq k$.
Applying Lemma~\ref{lemma.permutation-rep}, the coefficient of $m_\tau$ in $\mathbb{E}[\Exp_{\sigma}(X_H h_1)]$ is given by the number of orbits of $H$ acting on the set of ordered tuples $(T_1, \ldots, T_m)$ where each $T_i$ is a multiset of size $\tau_i$ of elements of $\{1, \ldots, n\}$.
As in the proof of \cite[Proposition~4.3.2]{Howe.RandomMatrixStatisticsAndZeroesOfLFunctionsViaProbabilityInLambdaRings}, giving such an ordered tuple is the same as giving a labeling of $\{1, \ldots, n\}$ by vectors in $\mathbb{Z}_{\geq 0}^m$ that adds up to $\tau$ --- the tuple of multisets $(T_1, \ldots, T_m)$ corresponds to the labeling where $i \in \{1, \ldots, n\}$ is labeled by the vector $\vec{v}_i$ whose $j$th entry is the multiplicity of $i$ in $T_j$.
Because $|\tau| \leq k$, there can be at most $k$ distinct nonzero labels appearing in any such labeling.
But then, because $H$ acts $k$-transitively, the $H$-orbit of the labeling depends only on the labels counted with multiplicity.
In other words, as in \cite[Proposition~4.3.2]{Howe.RandomMatrixStatisticsAndZeroesOfLFunctionsViaProbabilityInLambdaRings}, the number of orbits is equal to the number of $n$-part partitions of $\tau$ in $\mathbb{Z}_{\geq 0}^m$,
\[ \tau = \sum_{\vec{v} \in \mathbb{Z}_{\geq 0}^m} a_{\vec{v}} \vec{v}, \textrm{ where } a_{\vec{v}} \in \mathbb{Z}_{\geq 0} \textrm{ and } \sum_{\vec{v}} a_{\vec{v}}=n. \]
In particular, since $|\tau|\leq k \leq n$, any partition of $\tau$ in $\mathbb{Z}_{\geq 0}^m$ has at most $n$ vectors $\vec{v}$ with $a_{\vec{v}}\neq 0$, so this equals the number of partitions of $\tau$ in $\mathbb{Z}_{\geq 0}^m \backslash \{\vec{0}\}$,
\[ \tau = \sum_{\vec{v} \in \mathbb{Z}_{\geq 0}^m \backslash \{0\}} a_{\vec{v}} \vec{v}, \textrm{ where } a_{\vec{v}} \in \mathbb{Z}_{\geq 0}. \]
This agrees with the coefficient of $m_\tau$ in 
\begin{equation}\label{eq.generalized-partition-gen-function}
\Exp_{\sigma}(h_1 + h_2 + \ldots) = \Exp_{\sigma}\Biggl(\sum_{\vec{v} \in \mathbb{Z}_{\geq 0}^{\oplus \mathbb{N}}\backslash \vec{0}} \underline{t}^{\vec{v}}\Biggr) = \prod_{\vec{v} \in \mathbb{Z}_{\geq 0}^{\oplus \mathbb{N}}\backslash \vec{0}}\frac{1}{1-\underline{t}^{\vec{v}}}= \prod_{\vec{v} \in \mathbb{Z}_{\geq 0}^{\oplus \mathbb{N}}\backslash \vec{0}}\sum_{a \geq 0} \underline{t}^{a\vec{v}} 
\end{equation}
where, for $\vec{v}(j)$ the $j$th component of $\vec{v}$, $\underline{t}^{\vec{v}}:=t_1^{\vec{v}(1)}t_2^{\vec{v}(2)}\cdots$. 

We now consider the general case, $G= \mu_r^n \rtimes H$ with $r \geq 1$.
Since any element of $\mathrm{Sym}^\tau \mathbb{C}^n$ fixed by $G$ is fixed by $H$, it suffices to understand which of the orbits as above give rise to vectors also fixed by $\mu_r^n$.
But, on the basis vector in $\mathrm{Sym}^\tau \mathbb{C}^n$ corresponding to a labeling by $\vec{v}_1, \ldots, \vec{v}_n$, $(\zeta_1, \ldots, \zeta_n) \in \mu_r^n$ acts as multiplication by \begin{equation}\label{eq.rou-action} \prod_{i=1}^n \zeta_i^{\sum_j \vec{v}_i(j)}.\end{equation}
This scalar is $1$ for all elements in $\mu_r^n$ if and only if $r | \sum_j \vec{v}_i(j)$ for each $i$.
It then follows that the orbits corresponding to vectors fixed by $G$ are those corresponding to $\tau = \sum a_{\vec{v}} \vec{v}$ where each $\vec{v}$ with $a_{\vec{v}} \neq 0$ satisfies $r | \sum_j \vec{v}(j)$ (cf.\ the proof of Lemma~\ref{lemma.permutation-rep}).
At the level of the generating function \eqref{eq.generalized-partition-gen-function}, this amounts to only considering those monomials of degree divisible by $r$ in the indexing set, so we obtain
\[ \mathbb{E}[\Exp_{\sigma}(X_G h_1)]=\Exp_{\sigma}(h_r + h_{2r} + \ldots). \]

\end{proof}

\subsubsection{$W(D_n)$}
\begin{proof}[Proof of Theorem~\ref{thm.asymptotic-distributions-of-finite-groups}-(2)]
We have $W(D_n) = W \rtimes \Sigma_n$, where $W \leq (\mu_2)^n$ is the subgroup consisting of tuples whose product is $1$ (i.e., such that there are an even number of components equal to $-1$).
The analysis is thus the same as in the proof of Theorem~\ref{thm.asymptotic-distributions-of-finite-groups}-(1), except at the last step: we find now that the orbits corresponding to fixed vectors are exactly those where the quantities $\sum_j \vec{v}(j)$ are either all even (for $\vec{v}$ such that $a_{\vec{v}} \neq 0$) or all odd (where in the odd case there must be exactly $n$ non-zero vectors in the partition so that we get an odd exponent corresponding to each $\vec{v}_i$, $1 \leq i \leq n$, in \eqref{eq.rou-action}). The even case gives the contribution $\Exp_{\sigma}(h_2 + h_{4} + \ldots)$, and the odd case, under the constraint $|\tau| \leq n$, gives the contribution $h_n$ (there is one way, up to ordering, to write a monomial of degree $n$ as a product of $n$ monomials of degree $1$, and there are zero ways for a monomial of degree $<n$). 
\end{proof}

\section{\texorpdfstring{$\Lambda$}{\textLambda}-distributions and classical distributions}\label{s.lambda-to-classical}

In this section we establish a simple relationship between the shape of the $\sigma$-moment generating function or falling $\sigma$-moment generating function of a pre-$\lambda$ random variable and the independence of certain associated classical random variables (Theorem~\ref{thm.lambda-to-classical} and Corollary~\ref{cor.independence-mgf}). In Examples~\ref{example.orthogonal-and-symplectic} and \ref{example.poisson}, we use this to make a robust connection between the $\Lambda$-distributions for classical groups computed in \cite[Theorem~A]{Howe.RandomMatrixStatisticsAndZeroesOfLFunctionsViaProbabilityInLambdaRings} and the asymptotically independent random variables computed previously by Diaconis and Shahshahani \cite{DiaconisShahshahani.OnTheEigenvaluesOfRandomMatrices}; this was previously explained in \cite{Howe.RandomMatrixStatisticsAndZeroesOfLFunctionsViaProbabilityInLambdaRings} only by a kludgy direct comparison of complicated joint moment formulae.

In these and other applications, such comparisons are most often relevant not for a fixed pre-$\lambda$ random variable, but rather for the limiting distribution of a sequence of pre-$\lambda$ random variables. Thus, before, establishing the relation, we first establish some basic language for discussing abstract $\Lambda$-distributions and associated abstract classical (joint) distributions.

\subsection{Abstract \texorpdfstring{\boldmath$\Lambda$}{\textLambda}-distributions}
\begin{definition}
Let $C$ be a ring. An \emph{abstract $C$-valued $\Lambda$-distribution} is a homomorphism of abelian groups $\mu: \Lambda \rightarrow C$ such that $\mu(1_{\Lambda})=1_C$. The associated abstract $\sigma$-moment generating function is 
\[ \mathrm{SMG}_\mu := \mu(\Exp_{\sigma}(h_1(s_\bullet)h_1(t_\bullet))) = \sum_\tau \mu(h_\tau) m_\tau \in \Lambda_{C}^\wedge  \]
where in the middle we are working in the pre-$\lambda$ ring $\Lambda_{\Lambda}^\wedge$ with the $s_\bullet$ variables appearing in the coefficient copy of $\Lambda$, and $\mu$ is applied to these coefficients. 
Similarly, the associated abstract falling $\sigma$-moment generating function is 
\[ \mathrm{FSMG}_\mu=\mu( (1+h_1(t_\bullet))^{h_1(s_\bullet)}) = \sum_{\tau} \mu(c_\tau)m_\tau \in \Lambda_C^\wedge.\]
\end{definition}

\begin{example}
    If $(R, \mathbb{E}: R \rightarrow C)$ is a pre-$\lambda$ probability space, $X \in R$, and $\mu=\mu_X: \Lambda \rightarrow C$ is the $\Lambda$-distribution of the random variable $X$, \eqref{eq.smg-monomial} and \eqref{eq.fmg-monomial} imply
    \[ \mathrm{SMG}_\mu=\mathbb{E}[\Exp_{\sigma}(Xh_1)] \textrm{ and } \mathrm{FSMG}_\mu = \mathbb{E}[(1+h_1)^X] . \]
\end{example}

\subsubsection{}We note that, if $C$ is a $\mathbb{Q}$-algebra, then any abstract $C$-valued $\Lambda$-distribution $\mu$ extends uniquely to a $\mathbb{Q}$-linear map $\Lambda_{\mathbb{Q}} \rightarrow C$. By abuse of notation, we denote this extension also by $\mu$. 

\begin{definition}\label{def.abstr-ind-class}
Suppose $C$ is a $\mathbb{Q}$-algebra and $\mu$ is an abstract $C$-valued $\Lambda$-distribution.
\begin{enumerate} 
\item For $f \in \Lambda_{\mathbb{Q}}$, the \emph{moment generating function of the abstract classical distribution associated to $f$} is 
\[ \sum_{k \geq 0} \frac{\mu(f^k)}{k!} t^k \in 1+ tC[[t]].\]
The \emph{falling moment generating function of the abstract classical distribution associated to $f$} is 
\[ \sum_{k \geq 0} \mu\left(\binom{f}{k}\right) t^k = \sum_{k \geq 0} \frac{\mu(f(f -1)\cdots(f-k+1))}{k!}t^k \in 1+ tC[[t]].\]
\item Let $\{f_i\}_{i \in I}$ be elements of $\Lambda_{\mathbb{Q}}$. We say \emph{the associated abstract classical joint distribution of $\{f_i\}_{i \in I}$ is independent} if, for every finite set $\{i_1, \ldots, i_k\} \subseteq I$ and $g_{i_1}, \ldots, g_{i_k} \in \mathbb{Q}[x]$, 
\[ \mu(g_{i_1}(f_{i_1}) \cdot \ldots \cdot g_{i_k}(f_{i_k}) ) = \prod_{j=1}^k \mu(g_{i_j}(f_{i_j})). \]
\end{enumerate}
\end{definition}

\begin{remark}\label{remark.independence}
    Since $\mu$ is $\mathbb{Z}$-linear, in Definition~\ref{def.abstr-ind-class}-(2) it is equivalent to require the condition only for monomials, i.e., that for every $n_1,\ldots, n_k \in \mathbb{Z}_{\geq 0}$, 
\[ \mu(f_{i_1}^{n_1} \cdot \ldots \cdot f_{i_k}^{n_k})=\prod_{j=1}^k\mu(f_{i_j}^{n_j}), \]
i.e., that the abstract joint moments are multiplicative. Similarly, we could replace monomials with any $\mathbb{Q}$-basis of $\mathbb{Q}[x]$. This applies, in particular, to the binomial coefficients $\binom{x}{n}=\frac{(x)(x-1)\ldots(x-n+1)}{n!}$, in which case the condition becomes multiplicativity of joint falling moments.  
\end{remark}

The following example motivates Definition~\ref{def.abstr-ind-class}. 

\begin{example}Suppose $X_G$ is the $\mathbb{Z}[\mathbb{C}]$-valued random variable associated to a finite matrix group $G$ as in \S\ref{ss.matrix-group-dist} and $\mu=\mu_{X_G}$. Then, for any $f \in \Lambda$, 
\[ \mu(f^k) = \mathbb{E}[ f^k \circ X_G] = \mathbb{E}[ (f \circ X_G)^k]. \]
Unwinding the definitions, this is the $k$th moment of the classical $\mathbb{C}$-valued random variable $X_{G,f}$ sending $g \in G$ to the value of $f$ on the eigenvalues of $g$. In particular, the moment generating function of the abstract classical distribution associated to $f$ and $\mu$ in the sense of Definition~\ref{def.abstr-ind-class}-(1) \emph{is} the moment generating function of the genuine classical random variable $X_{G,f}$, and similarly for the falling moment generating function. Since $X_{G,f}$ is bounded, if it is real-valued then these moment generating functions uniquely determine the distribution of $X_{G,f}$.

Moreover, given a collection of symmetric functions $\{f_i\}_{i \in I}$ such that each $X_{G,f_i}$ is real-valued, the associated abstract classical joint distribution of $\{f_i\}_{i \in I}$ is independent in the sense of Definition~\ref{def.abstr-ind-class}-(2) if and only if the genuine classical random variables $\{X_{G, f_i}\}_{i \in I}$ are independent (here we use that these classical random variables are bounded so that independence is equivalent to multiplicativity in the formation of joint moments). 
\end{example}

\begin{remark}
    If $\mu$ is an abstract $C$-valued $\Lambda$-distribution where $C$ is torsion-free, then there is no loss of information in replacing $C$ with $C \otimes_{\mathbb{Z}} \mathbb{Q}$ to apply the results and definitions here and in what follows. Note that if $C$ is a torsion-free pre-$\lambda$ ring then there is a canonical extension of the pre-$\lambda$ ring structure to $C \otimes_{\mathbb{Z}} \mathbb{Q}$. 
\end{remark}

\subsection{The shapes of independence}

\begin{maintheorem}\label{thm.lambda-to-classical}
  Suppose $C$ is a $\mathbb{Q}$-algebra and $\mu$ is an abstract $C$-valued $\Lambda$-distribution.
  \begin{enumerate}
  \item The associated abstract classical joint distribution of the power sum symmetric functions $\{p_i\}_{i \in \mathbb{Z}_{\geq 1}}$ is independent (Definition~\ref{def.abstr-ind-class}) if and only if there are power series $\gamma_i \in 1+t C[[t]]$ such that
\[ \mathrm{SMG}_{\mu}= \prod_{i=1}^\infty \gamma_i\left(\frac{p_i}{i}\right). \]
    Moreover, in this case $\gamma_i$ is the moment generating function of the abstract classical distribution associated to $p_i$. 
    
  \item The associated abstract classical joint distribution of the cycle counting symmetric functions $\{p'_i\}_{i \in \mathbb{Z}_{\geq 1}}$ is independent if and only if there are power series $\gamma_i \in 1 + t\mathbb{C}[[t]]$ such that
\[ \mathrm{FSMG}_{\mu}= \prod_{i=1}^\infty \gamma_i(p_i). \]
    Moreover, in this case $\gamma_i$ is the falling moment generating function of the abstract classical distribution associated to $p_i'$. 
  \end{enumerate}
\end{maintheorem}
\begin{proof}
To prove part (1), first recall that 
\[ \mathrm{SMG}_\mu := \mu(\Exp_{\sigma}(h_1(s_\bullet)h_1(t_\bullet))). \]
By \eqref{eq.exp-power-sum-expansion}, we may rewrite this as
\[ \mathrm{SMG}_\mu = \mu\left( \prod_{i=1}^\infty \mathrm{exp}\left(\frac{1}{i}p_i \circ (h_1(s_\bullet)h_1(t_\bullet))\right)\right)=\mu \left(\prod_{i=1}^\infty \mathrm{exp}\left(\frac{1}{i}p_i(s_\bullet)p_i(t_\bullet)\right) \right),   \]
where to obtain the last equality we have used the definition of the $\lambda$-ring structure on $\Lambda_{\Lambda}^\wedge \subseteq \Lambda[[t_{\mathbb{N}}]]$ as in \cite[Lemma~2.2.4, \S2.3]{Howe.RandomMatrixStatisticsAndZeroesOfLFunctionsViaProbabilityInLambdaRings}. Expanding, we obtain 
\[ \mathrm{SMG}_\mu = \sum_{(k_1, k_2, \ldots)} \frac{1}{k_1!k_2!\cdots} \mu\left(p_1^{k_1}p_2^{k_2} \cdots\right)\left(\frac{p_1}{1}\right)^{k_1}\left(\frac{p_2}{2}\right)^{k_2} \cdots. \]
Since $\Lambda_C^\wedge$ is an infinite formal power series ring in the variables $\frac{p_i}{i}$, we find this is of the form
\[ \prod_{i=1}^\infty \gamma_i\left(\frac{p_i}{i}\right) \]
if and only if 
\begin{equation}\label{eq.first-crv-expansion} \gamma_i=\sum_{k = 0}^\infty\frac{ \mu(p_i^{k})}{k!}t^k \end{equation}
and 
\[ \mu(p_1^{k_1}p_2^{k_2} \cdots)=\mu(p_1^{k_1})\mu(p_2^{k_2}) \cdots, \]
i.e., if and only if $\gamma_i$ is the moment generating function of the abstract classical distribution associated to $p_i$ and the associated abstract classical joint distribution of the $\{p_i\}_{i \in \mathbb{Z}_{\geq 1}}$ is independent (using Remark~\ref{remark.independence}). 

We argue part (2) similarly: first recall that
\[ \mathrm{FSMG}_\mu=\mu( (1+h_1(t_\bullet))^{h_1(s_\bullet)}).  \]
By \eqref{eq.power-expansion}, we may rewrite this as
\[ \mathrm{FSMG}_{\mu}=\mu \left(\prod_{i=1}^\infty (1 + h_1(t_\bullet^i))^{p_i' \circ h_1(s_\bullet)}\right)=\mu \left(\prod_{i=1}^\infty (1 + p_i(t_\bullet)) ^{p_i'(s_\bullet)}\right)\]
where the powers in this equation are not taken in the pre-$\lambda$ sense but are instead given by the naive binomial expansion 
\[ (1+a)^N=\sum_{k \geq 0} \binom{N}{k} a^k=\sum_{k \geq 0} \frac{N(N-1)\cdots(N-k+1)}{k!} a^k.\]
Expanding, we obtain
\[ \mathrm{FSMG}_{\mu}= \sum_{(k_1, k_2, \ldots)} \mu\left(\binom{p_1'}{k_1}\binom{p_2'}{k_2}\cdots\right) p_1^{k_1}p_2^{k_2} \cdots. \]
Since $\Lambda_C^\wedge$ is an infinite formal power series ring in the variables $p_i$,
we find this is of the form
\[ \prod_{i=1}^\infty \gamma_i(p_i) \]
if and only if 
\begin{equation}\label{eq.second-crv-expansion} \gamma_i=\sum_{k = 0}^\infty\mu\left(\binom{p_i'}{k}\right)t^k \end{equation}
and 
\[ \mu\left(\binom{p_1'}{k_1}\binom{p_2'}{k_2}\cdots\right)=\mu\left(\binom{p_1'}{k_1}\right)\mu\left(\binom{p_2'}{k_2}\right)\cdots, \]
i.e., if and only if $\gamma_i$ is the falling moment generating function of the classical distribution associated to $p_i'$ and the associated classical joint distribution of the $\{p_i'\}_{i \in \mathbb{Z}_{\geq 1}}$ is independent (using Remark~\ref{remark.independence}). 
\end{proof}

In many natural examples the $\sigma$-moment or falling $\sigma$-moment generating function has a simple expression using $\Exp_\sigma$. The following corollary deduces from Theorem~\ref{thm.lambda-to-classical} an easy-to-apply criterion for the independence of associated classical joint distributions in terms of such an expression.

\begin{corollary}\label{cor.independence-mgf}
Let $\mu$ be an abstract $C$-valued $\Lambda$-distribution where $C$ is a $\mathbb{Q}$-algebra and a pre-$\lambda$ ring.
\begin{enumerate}
\item  If $\mathrm{SMG}_{\mu}= \Exp_{\sigma}(\sum_{i,k \geq 1} a_{i,k} p_i^k)$ for $a_{i,k} \in C$, then 
the associated abstract classical joint distribution of the power sum symmetric functions $\{p_i\}_{i \in \mathbb{Z}_{\geq 1}}$ is independent (Definition~\ref{def.abstr-ind-class}).  
\item If $\mathrm{FSMG}_{\mu}=\Exp_{\sigma}(\sum_{i,k \geq 1} a_{i,k} p_i^k)$ for $a_{i,k} \in C$, then the associated abstract classical joint distribution of the cycle counting symmetric functions $\{p'_i\}_{i \in \mathbb{Z}_{\geq 1}}$ is independent.
\end{enumerate}
\end{corollary}
\begin{proof}
Using \eqref{eq.exp-power-sum-expansion} for the second equality and \cite[Lemma~2.4]{Howe.RandomMatrixStatisticsAndZeroesOfLFunctionsViaProbabilityInLambdaRings} for the third, we find
    \begin{align*} \Exp_{\sigma}\left(\sum_{i,k \geq 1} a_{i,k} p_i^k\right) & = \prod_{i,k \geq 1} \Exp_{\sigma}(a_{i,k} p_i^k) \\
    & = \prod_{i,j,k \geq 1}\exp\left(\frac{1}{j}p_j \circ  (a_{i,k}p_i^k)\right)\\
    &=\prod_{i,j,k \geq 1}\exp\left(\frac{1}{j}(p_j \circ  a_{i,k})p_{ji}^k\right)\\
    &=\prod_{n\geq 1}\exp\left(\sum_{i,j,k \geq 1, ji=n} \frac{1}{j}(p_j \circ  a_{i,k})p_{n}^k\right).
    \end{align*}
As the $n$th term in the product expands to a power series in $p_n$, Theorem~\ref{thm.lambda-to-classical} applies. 
\end{proof}

\begin{remark}
In other words, Corollary~\ref{cor.independence-mgf} says that one obtains a classical independence result whenever $\Log_{\sigma}(\mathrm{SMG}_{\mu})$ or $\Log_{\sigma}(\mathrm{FSMG}_{\mu})$ is a formal sum of multiples of powers of power sum symmetric functions --- note that we can always expand these uniquely as formal series in the $p_\tau$ as $\tau$ varies over all partitions, but such an expansion can in general also include mixed terms such as $p_{(2,1,0,\ldots)}=p_2p_1$. 
\end{remark}

\begin{remark}
    The proofs in the two cases of Theorem~\ref{thm.lambda-to-classical} and Corollary~\ref{cor.independence-mgf} are similar, and it is possible to give an abstract generalization of both cases that applies to more general types of moment generating functions with a uniform proof. However, because only $\sigma$-moment generating functions and falling $\sigma$-moment generating functions seem to arise naturally in examples and applications, we have preferred to give instead the concrete statements and arguments in these two cases. 
\end{remark}

\subsection{Examples}

\begin{example}\label{example.orthogonal-and-symplectic} For $\mathrm{O}(n)\subseteq \mathrm{GL}_n(\mathbb{R})$ the $n \times n$ orthogonal group, we consider the random variable $X_{\mathrm{O}(n)}$ sending $M \in \mathrm{O}(n)$ to its multiset of eigenvalues (cf. \S\ref{ss.matrix-group-dist}) and write its $\Lambda$-distribution as $\mu_{\mathrm{O}(n)}: \Lambda \rightarrow \mathbb{Z}$. By \cite[Theorem~A]{Howe.RandomMatrixStatisticsAndZeroesOfLFunctionsViaProbabilityInLambdaRings}, 
the limit as $n \rightarrow \infty$ of $\mu_{\mathrm{O}(n)}$ is an abstract $\Lambda$-distribution $\mu_{\mathrm{O}(\infty)}$ determined by 
\[ \mathrm{SMG}_{\mu_{\mathrm{O}(\infty)}} = \Exp_{\sigma}(h_2)=\Exp_{\sigma}\left(\frac{1}{2}p_1^2 + \frac{1}{2}p_2\right) .\]
Applying Corollary~\ref{cor.independence-mgf}, we see immediately that the abstract classical random variables corresponding to the $p_i$ are independent; equivalently, the random variables $M \mapsto \mathrm{Tr}(M^i)$ in $\mathrm{O}(n)$ are asymptotically independent as $n \rightarrow \infty$. Moreover, the explicit limiting Gaussian distributions of these random variables described by Diaconis and Shahshahani in \cite[Theorem~4]{DiaconisShahshahani.OnTheEigenvaluesOfRandomMatrices} (see also \cite[Remarks~1.1.6 and 4.4.2]{Howe.RandomMatrixStatisticsAndZeroesOfLFunctionsViaProbabilityInLambdaRings}) are obtained from Theorem~\ref{thm.lambda-to-classical} by expanding as in the proof of Corollary~\ref{cor.independence-mgf}:
\begin{multline*}  \Exp_{\sigma}\left(\frac{1}{2}p_1^2 + \frac{1}{2}p_2\right) =\\\prod_{i\geq 1}\left( \begin{cases} \exp\left( \frac{1}{2} \frac{1}{i} p_i^2 \right) = \gamma_i(\frac{1}{i}p_i),\; \gamma_i(t):=\exp\left(\frac{1}{2} it^2\right) &  \textrm{$i$ odd} \\
 \exp\left(\frac{1}{i}p_{i} + \frac{1}{2} \frac{1}{j} p_i^2 \right)=\gamma_i(\frac{1}{i}p_i),\; \gamma_i(t):=\exp\left(t+\frac{1}{2} it^2\right) & \textrm{$i$ even}
\end{cases} \right).\end{multline*}
Indeed, $\gamma_i$ is the moment generating function of a Gaussian of variance $i$ and mean $0$ or $1$ as $i$ is odd or even. By the same method, one obtains the asymptotic independence and distributions of the traces of random matrices in the compact symplectic group $\mathrm{Sp}(n)$, $n \rightarrow \infty$, starting from the result of \cite[Theorem~A]{Howe.RandomMatrixStatisticsAndZeroesOfLFunctionsViaProbabilityInLambdaRings} which, in notation mirroring that for the orthogonal groups used above, says
\[ \mathrm{SMG}_{\mu_{\mathrm{Sp}(\infty)}} = \Exp_{\sigma}(e_2)=\Exp_{\sigma}\left(\frac{1}{2}p_1^2 - \frac{1}{2}p_2\right).\]
Expanding, we find the asymptotic moment generating function for  $M \mapsto \mathrm{Tr}(M^i)$ is that of a Gaussian of variance $i$ and mean $0$ or $-1$ as $i$ is odd or even.

\end{example}

\begin{example}\label{example.poisson}
    For $C$ a pre-$\lambda$ ring, a $C$-valued $\Lambda$-distribution $\mu$ is \emph{an abstract Poisson distribution of mean $m \in C$}  if and only if the following equivalent (by \cite[Lemma~3.3.1-(2)]{Howe.RandomMatrixStatisticsAndZeroesOfLFunctionsViaProbabilityInLambdaRings}) conditions hold:
    \begin{align*} \mathrm{SMG}_{\mu} &= \Exp_{\sigma}(m(h_1 + h_2 + h_3 + \ldots)) \textrm{ and } \\
    \mathrm{FSMG}_{\mu}&=\Exp_{\sigma}(mh_1). \end{align*}
    The form of the $\sigma$-moment generating function combined with Theorem~\ref{thm.asymptotic-distributions-of-finite-groups}-(1) shows that, for any sequence $G_n$ of $k_n$-transitive subgroups of $\Sigma_n$ with $k_n \rightarrow \infty$ as $n \rightarrow \infty$, the limiting $\Lambda$-distribution $\mu_{G_\infty}= \lim_{n \rightarrow \infty} \mu_{G_n}$ is an abstract Poisson distribution of mean $1$. This applies, in particular, to $G_n=\Sigma_n$ (see also \cite[Theorem~A-(3)]{Howe.RandomMatrixStatisticsAndZeroesOfLFunctionsViaProbabilityInLambdaRings}) or $G_n=A_n$, the sequence of alternating groups as in Example~\ref{example.alternating}. 
    
    Now, if $\mu$ is Poisson, the formula for the falling $\sigma$-moment generating function can be rewritten as
 \[ \Exp_{\sigma}(mp_1)=\prod_{i=1}^\infty \exp\left(\frac{1}{i}p_i\circ (mp_1)\right)=\prod_{i=1}^\infty \exp\left(\frac{1}{i} (p_i \circ m) p_i\right) \]
 where we have used \eqref{eq.exp-power-sum-expansion} in the first equality and \cite[Lemma~2.4]{Howe.RandomMatrixStatisticsAndZeroesOfLFunctionsViaProbabilityInLambdaRings} in the second equality. Theorem~\ref{thm.lambda-to-classical} then gives that the abstract classical joint distribution associated to the cycle-counting symmetric functions $p_i'$ is independent, and the $i$th one has an abstract classical Poisson distribution of mean $\frac{1}{i} p_i \circ m$. 
 
 In particular, for a sequence of sufficiently transitive permutation groups $G_n\leq \Sigma_n$ as above, we find that the random variables counting $i$-cycles on elements of $G_n$ are asymptotically (as $n \rightarrow \infty$) independent Poisson random variables of mean $\frac{1}{i}$. In the case of the full symmetric group, this is \cite[Theorem~7]{DiaconisShahshahani.OnTheEigenvaluesOfRandomMatrices} (see also \cite[\S4.5]{Howe.MotivicRandomVariablesAndRepresentationStabilityIConfigurationSpaces}).
\end{example}

\bibliographystyle{plain}
\bibliography{references}

@article{Howe.RandomMatrixStatisticsAndZeroesOfLFunctionsViaProbabilityInLambdaRings,
title={Random matrix statistics and zeroes of {L}-functions via probability in $\lambda$-rings},
author={Howe, Sean},
journal={ar{X}iv:2412.19295},
year={2024}
}

@article{BertucciHowe.EquidistributionAndArithmeticLambdaDistributions,
title={Equidistribution and arithmetic {$\Lambda$}-distributions},
author={Bertucci, Matthew and Howe, Sean},
journal={ar{X}iv:2505.24748},
year={2025}
}

@article{Howe.TheNegativeSigmaMomentGeneratingFunction,
title={The negative $\sigma$-moment generating function},
author={Howe, Sean},
journal={ar{X}iv:2505.01205},
year={2025}
}

@article {Howe.MotivicRandomVariablesAndRepresentationStabilityIConfigurationSpaces,
    AUTHOR = {Howe, Sean},
     TITLE = {Motivic random variables and representation stability {I}:
              {C}onfiguration spaces},
   JOURNAL = {Algebr. Geom. Topol.},
  FJOURNAL = {Algebraic \& Geometric Topology},
    VOLUME = {20},
      YEAR = {2020},
    NUMBER = {6},
     PAGES = {3013--3045},
      ISSN = {1472-2747,1472-2739},
   MRCLASS = {14C35 (14G10 18F30 55R80)},
  MRNUMBER = {4185934},
MRREVIEWER = {Geoffrey\ M. L. Powell},
       DOI = {10.2140/agt.2020.20.3013},
       URL = {https://doi.org/10.2140/agt.2020.20.3013},
}

@article {Stanley.InvariantsOfFiniteGroupsAndTheirApplicationsToCombinatorics,
    AUTHOR = {Stanley, Richard P.},
     TITLE = {Invariants of finite groups and their applications to
              combinatorics},
   JOURNAL = {Bull. Amer. Math. Soc. (N.S.)},
  FJOURNAL = {American Mathematical Society. Bulletin. New Series},
    VOLUME = {1},
      YEAR = {1979},
    NUMBER = {3},
     PAGES = {475--511},
      ISSN = {0273-0979,1088-9485},
   MRCLASS = {20C15 (05A15 13H10 14H10 15A72 51F15)},
  MRNUMBER = {526968},
MRREVIEWER = {Ralph\ Strebel},
       DOI = {10.1090/S0273-0979-1979-14597-X},
       URL = {https://doi.org/10.1090/S0273-0979-1979-14597-X},
}

@article{DiaconisShahshahani.OnTheEigenvaluesOfRandomMatrices,
  title={On the eigenvalues of random matrices},
  volume={31},
  DOI={10.1017/S0021900200106989},
  number={A},
  fjournal={Journal of Applied Probability},
  journal    = {J. Appl. Probab.},
  author={Diaconis, Persi and Shahshahani, Mehrdad},
  year={1994},
  pages={49--62},
  url        = {https://doi.org/10.2307/3214948},
}

\end{document}